\documentclass[11pt,reqno,final]{amsart}

\usepackage[scaled]{helvet} 
\usepackage{fourier} 
\usepackage[usenames,dvipsnames,svgnames,table]{xcolor}
\usepackage{color}
\usepackage{url,graphicx,tabularx,array,geometry, amsthm, amsfonts,parskip,alltt, amsmath,tikz}
\usepackage{hyperref}

\usepackage{url,graphicx,tabularx,array,geometry, amsthm, amsfonts,parskip,alltt, amsmath,hyperref,xcolor}
\usepackage{mathtools,tikz,algorithm,algpseudocode,amssymb}
\usepackage{pgf,tikz}
\usepackage{mathrsfs}
\usetikzlibrary{arrows}

\usepackage[export]{adjustbox}
\usepackage{subcaption}
\usepackage{svg}
\usepackage{comment}

\newtheorem{theorem}{Theorem}
\newtheorem{proposition}[theorem]{Proposition}

\usetikzlibrary{arrows}

\newcommand{\ve}[1]{\mathbf{#1}}
\newcommand\norm[1]{\left\lVert#1\right\rVert}

\makeatletter
\renewcommand*{\ALG@name}{Method}
\makeatother\usepackage[utf8]{inputenc}

\title{On Application of Block Kaczmarz Methods in Matrix Factorization}
\author[Edwin Chau]{Edwin Chau \\\\ Sponsor: Jamie Haddock}

\begin{document}

\maketitle

\begin{abstract}
Matrix factorization techniques compute low-rank product approximations of high dimensional data matrices and as a result, are often employed in recommender systems and collaborative filtering applications. However, many algorithms for this task utilize an exact least-squares solver whose computation is time consuming and memory-expensive. In this paper we discuss and test a block Kaczmarz solver that replaces the least-squares subroutine in the common alternating scheme for matrix factorization.  This variant trades a small increase in factorization error for significantly faster algorithmic performance. In doing so we find block sizes that produce a solution comparable to that of the least-squares solver for only a fraction of the runtime and working memory requirement.
\end{abstract}

\section{Introduction}

With the recent rise and ubiquity of online services, the volume of data available for and demanding analysis has exploded. This data often resides in extremely high-dimensional space, and for this reason, it is often computationally and intuitively beneficial to reduce its dimension, offering decreased required storage space and information about the latent trends within the data. Matrix factorization (e.g., principal component analysis/singular value decomposition (PCA/SVD)~\cite{pearson1901liii}, CUR factorization~\cite{mahoney2009cur}, and the nonnegative matrix factorization (NMF)~\cite{leeseung}) are common approaches 
to find factor matrices whose product is a low-rank approximation to the original data matrix. These factor matrices, when combined with auxiliary methods, can help identify useful trends and features present in the data. As a result, matrix factorization models often find use in tasks such as collaborative filtering~\cite{he, pan, zhou} and topic modeling~\cite{berryemail}. Specific applications in which these models have been applied are document classification~\cite{gaussier2005relation}, facial recognition~\cite{turk1991face}, and the winning submission to the well-known Netflix Prize competition applied these methods to collaborative filtering for movie recommendation~\cite{koren2009collaborative}.

While many dimension-reduction and topic modeling approaches are framed as matrix factorization models, we focus on the simple low-rank factorization model for data $X \in \mathbb{R}^{m \times n}$ that seeks $A \in \mathbb{R}^{m \times k}$, generally called the dictionary matrix, and $S \in \mathbb{R}^{k \times n}$, generally called the representation matrix, such that $X \approx AS$.  Here we consider the \emph{factor rank} $k$ to be a user-defined parameter but assume that $k < \min\{m,n\}$ so  that the resulting factorization reduces the dimension of the original data and/or reveals latent themes in the data.  Generally, the factor rank $k$ can be chosen according to \emph{a priori} information or a heuristic method.  
Data points are typically stored as columns of $X$, thus $m$ represents the number of features, and $n$ represents the number of samples.
The columns of $A$ are generally referred to as \emph{topics} or \emph{dictionary elements}, which are characterized by features of the dataset.
Further, each column of $S$ provides the approximate representation of the respective column in $X$ in the lower-dimensional space spanned by the columns of $A$.
Thus, the data points are
well approximated by a linear combination of the latent topics with coefficients given by the entries of $S$.

Many methods for matrix factorization utilize an alternating scheme due to the non-convexity of some common problem formulations.  These methods iteratively alternate through the factor matrices, holding the others constant and updating one at a time. 
For example, the common Frobenius norm formulation of the matrix factorization problem asks, given a data matrix $X$, to minimize the objective function 
\begin{equation} 
    \norm{X - AS}_F^2 \label{eq:objfunction}
\end{equation}
with respect to $A \in \mathbb{R}^{m \times k}$ and $S \in \mathbb{R}^{k \times n}$
Alternating methods that approximately minimize this formulation do so by first fixing $A$ and iteratively solving for $S$ such that $X \approx AS$ (usually using the same formulation \eqref{eq:objfunction}), and then fixing $S$ and iteratively solving for $A$ such that $X \approx AS$. After fixing all but one of the factors, the resulting problem is often convex and the formulation encourages solution of a simple linear system; in the case of \eqref{eq:objfunction}, the resulting problems are simple least-squares problems. While the details of this process will be discussed in Section~\ref{subsec:preliminaries}, we note now that the quality and computational time of the solution is dictated by the chosen linear system solver.

One commonly utilized solver is the method of least-squares, where the exact least-squares solution (approximation to the linear system) is computed using the pseudo-inverse of the measurement matrix. However, while quick to converge, least-squares has a two-fold problem. First, calculating an exact solution for commonly encountered large linear systems is computationally expensive, requiring significantly more time to run as data matrices increase in dimension. Second, ALS requires the entire data matrix to be loaded into memory to solve for a solution, which significantly limits the size of systems for which this technique can be employed. 
Because of these drawbacks, we focus instead on utilizing an iterative method of solving linear systems within the alternating scheme of matrix factorization, namely \emph{randomized Kaczmarz (RK)} and its generalization, \emph{block randomized Kaczmarz (BRK)}. 

Kaczmarz methods are a large family of methods used to approximate the solution of a linear system by iteratively sampling and solving a subset of the system. The Kaczmarz method was first introduced by S. Kaczmarz~\cite{kaczmarzoriginal}. The randomized Kaczmarz (RK) method, a variant of the Kaczmarz method that randomly samples rows proportional to their $\ell_{2}$ norm, was later introduced by T. Strohmer and R. Vershynin and proven to have exponential convergence~\cite{strohmer}. The \emph{randomized extended Kaczmarz (REK)} is a variant of RK suitable for noisy systems and which converges to the least-squares solution~\cite{zouzias}. Furthermore, the generalization of RK known as block randomized Kaczmarz (BRK), where a submatrix is sampled rather than a single row, was introduced and shown to outpace RK on both well paved and randomly generated matrices~\cite{needell}. Employment of nonstandard stepsizes has further sped the convergence of the BRK method~\cite{necoara}.

Due to the often sparse and redundant nature of large scale data, a Kaczmarz method is a natural alternative to least-squares in a matrix factorization approach. We therefore consider an approach which replaces the traditional least-squares solver with BRK. This family of methods aims to approximate a solution to each linear system rather than computing an exact one, trading accuracy for a quicker runtime. 

We note that the methods studied in this paper are not novel; similar methods have been proposed in the context of sensor network localization where matrix partitions are fixed before sampling~\cite{gogna}. However, we demonstrate that this BRK-based approach performs well on a wider range of applications and also when the sampling procedure is more flexible. Our contributions include empirical results on sparse large-scale synthetic and real-world datasets as well as a publicly available Python package\footnote{\url{https://pypi.org/project/mf-algorithms/}} implementing the methods discussed.  

\subsection{Organization} \label{subsec:organization}
The rest of the paper is organized as follows. Section~\ref{sec:algorithms} provides details of the aforementioned ALS and BRK methods. Section~\ref{sec:experiments} applies these methods on small-scale synthetic, large-scale synthetic, and large-scale real-world datasets and compares their approximation error and runtimes. Section~\ref{sec:theory} provides theoretical insight on stationarity of the BRK-based matrix factorization method. Finally, in Section~\ref{sec:conclusions}, we provide some final conclusions and discussion of future work.

\subsection{Notation} \label{subsec:notation}
Let $A \in \mathbb{R}^{m \times n}$ be a matrix with $m$ rows and $n$ columns. For a matrix $A$, we denote the $i^{th}$ row of $A$ as $A_{i, :}$ and the $i^{th}$ column as $A_{:, i}$. We let $\norm{A}$ denote the operator norm of $A$, $\norm{A}_F$ denote the Frobenius norm of $A$, $A^{\dagger}$ denote the Moore-Penrose pseudoinverse of $A$, and $\sigma_{\text{max}}(A)$ denote the maximum singular value of $A$.  Finally, the methods considered in this work iteratively produce approximations to factor matrices $A$ and $S$ in the matrix factorization formulation \eqref{eq:objfunction}; we denote the approximation to matrix $A$ produced in the $j^{\text{th}}$ iteration of a method as $A^{(j)}$. We let $[n]$ denote the integers in the interval from $1$ to $n$, $[n] = \{1, 2, \ldots, n\}$. 

\subsection{Preliminaries} \label{subsec:preliminaries}
As mentioned previously, the alternating approach to computing the matrix factorization, 
\begin{equation} 
    AS \approx X, \label{eq:matrixeq}
\end{equation} 
is to repeatedly fix one factor matrix and update the other. A typical alternating scheme attempts to minimize the objective function \eqref{eq:objfunction} by reducing equation \eqref{eq:matrixeq} column-wise to linear systems 
\begin{equation}
   A S_{:, i} \approx X_{:, i} \text{ or } A_{j, :} S \approx X_{j, :} \label{eq:reduced}
\end{equation} 
if solving for $S$ or $A$, respectively. The alternating least-squares (ALS) method would then utilize the method of least-squares to update $S_{:, i}$ or $A_{j, :}$ before resampling. On the other hand, a Kaczmarz method takes a sample of the equations $\tau_1$ or $\tau_2$ and further reduces equations \eqref{eq:reduced} to 
\begin{equation}\label{eq:sublinsys}
    A_{\tau_1, :} S_{:, i} \approx X_{\tau_1, i} \text{ or } A_{j, :} S_{:, \tau_2} \approx X_{j, \tau_2}
\end{equation}
before updating approximate solutions for $S_{:, i}$ and $A_{j, :}$ as outlined in Section~\ref{subsec:brk}.  See Figure~\ref{fig:general} for a visualization of these algorithmic reductions.

\begin{figure}[tb]
    \makebox[\textwidth][c]{\includegraphics[width=\linewidth]{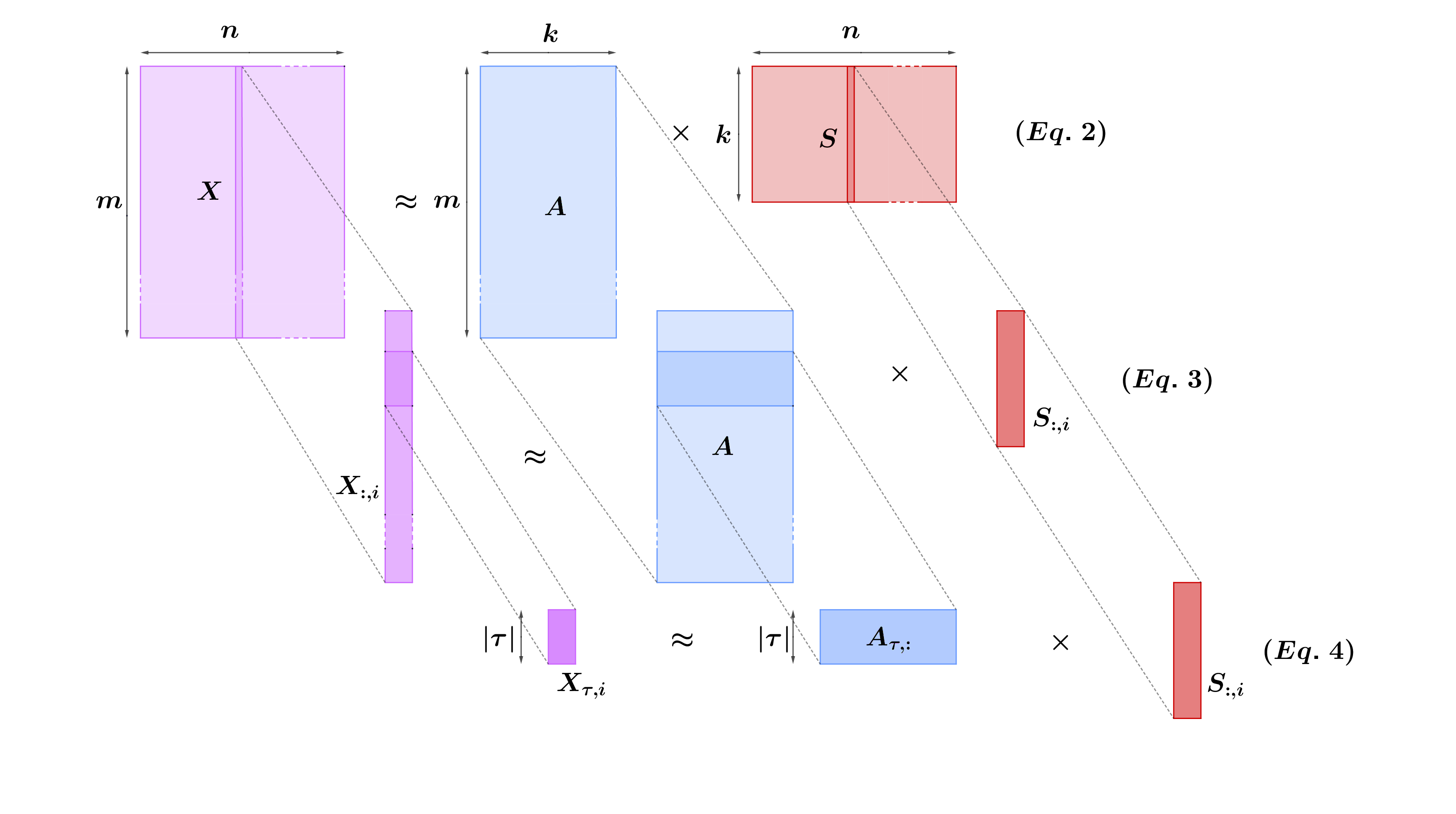}}
    \caption{The matrix equation (2) is reduced to a linear system (3) and solved using least-squares or further reduced to (4) and solved with an iterative method.}
    \label{fig:general}
\end{figure}

\section{Algorithms} \label{sec:algorithms}
The methods we consider follow the alternating scheme, which iteratively updates one of the factor matrices while holding the other constant.  They reduce the computation required in each iteration of alternating least-squares via the sequence of reductions visualized in Figure~\ref{fig:general}.  We refer to the first reduction from \eqref{eq:matrixeq} to \eqref{eq:reduced} as the \emph{matrix-vector reduction}, and the second reduction from \eqref{eq:reduced} to \eqref{eq:sublinsys} as the \emph{vector-sample reduction}.

The nature of the matrix-vector reduction affects the resulting solution. This reduction can be performed cyclically, where the algorithm iterates through and updates all rows of $A$ before iterating through and updating all columns of $S$. In this context, one iteration of the algorithm refers to one epoch, where rows and columns of $A$ and $S$ are updated once. 

The matrix-vector reductions can also be determined in a stochastic manner, where the algorithm alternates between updating one uniformly sampled row of $A$ and one uniformly sampled column of $S$. These are sampled without replacement to prevent solving the subsampled linear system so precisely that it reduces the quality of the overall matrix factorization approximation. In the case that $X \in \mathbb{R}^{m \times n}$ is not square, the number of row updates to $A$ and column updates to $S$ can be made proportional; that is, for every one column update to $S$, we perform $\lceil m/n \rceil$ row updates to $A$ (assuming $m > n$). 
Here, one epoch corresponds to $\text{min}(m, n)$ iterations.

The algorithms we will discuss and compare follow the alternating update scheme, the iterations of which we will refer to as "alternating" iterations. Our focus is on the BRK-based alternating method, where the vector-sample reduction is performed according to a stochastic sample. While algorithms could employ a combination of cyclic and stochastic schemes (for example, performing cyclic matrix-vector reductions and stochastic vector-sample reductions), they are beyond the scope of this paper. We focus instead only on a method which applies random sampling in both the matrix-vector and vector-sample reductions.

Finally, the manner of initialization of the factor matrices $A$ and $S$ also affects the final approximation. Because the objective function (\ref{eq:objfunction}) is non-convex in both $A$ and $S$~\cite{zhu2018optim}, there are many possible local optima within reach of matrix factorization algorithms and which the method finally reaches depends upon the initialization of the factor matrices. For this reason, our methods will simply randomly initialize two factor matrices before iteratively computing an improved approximation. While more sophisticated initialization techniques have been developed to promote faster convergence in specific applications, most do not guarantee this property in a general setting~\cite{gillis}.

\subsection{Alternating Least-Squares}
The ALS algorithm is as follows. Given a data matrix $X \in \mathbb{R}^{m \times n}$ and a target factor rank $k$, two factor matrices $A \in \mathbb{R}^{m \times k}$ and $S \in \mathbb{R}^{k \times n}$ are randomly initialized. It first samples an $i^{th}$ column of $S$ and performs a matrix-vector reduction. 

It then solves for $S_{:, i}$ using the least-squares method, updating $S_{:,i}$ with the explicit least-squares solution. The algorithm then repeats for a row of $A$. This iterative process continues until an exit condition is met. The pseudocode for this algorithm is provided in Method~\ref{alg:als}; note that this algorithm exploits only the matrix-vector reductions. Furthermore, ALS would specifically use the explicit least-squares solution as the LS solver.

\begin{figure}
    \centering
    \begin{subfigure}[b]{.5\textwidth}
      \centering
      \includegraphics[width=\linewidth]{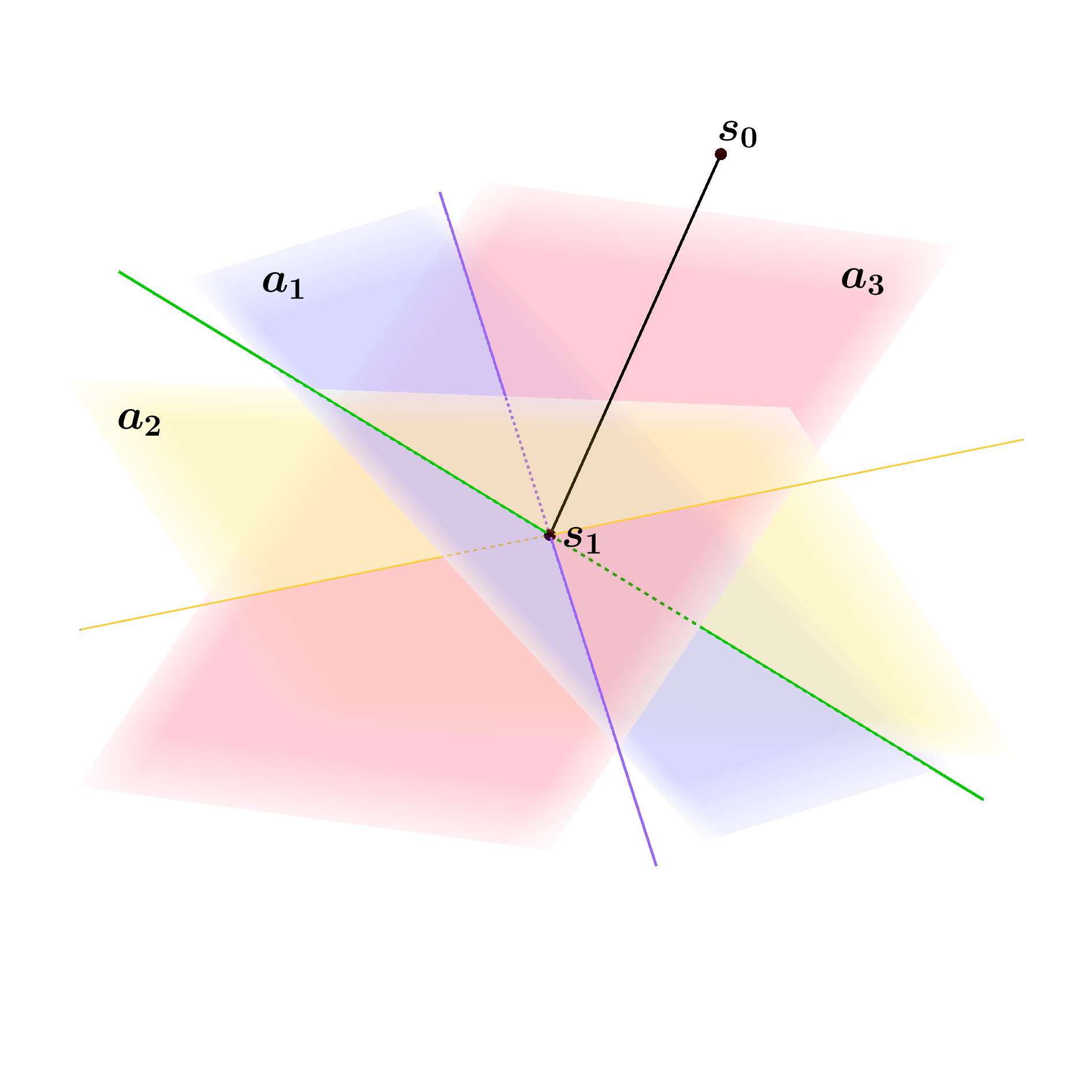}
      \caption{ALS update}
      \label{fig:updateals}
    \end{subfigure}%
    \begin{subfigure}[b]{.5\textwidth}
      \centering
      \includegraphics[width=\linewidth]{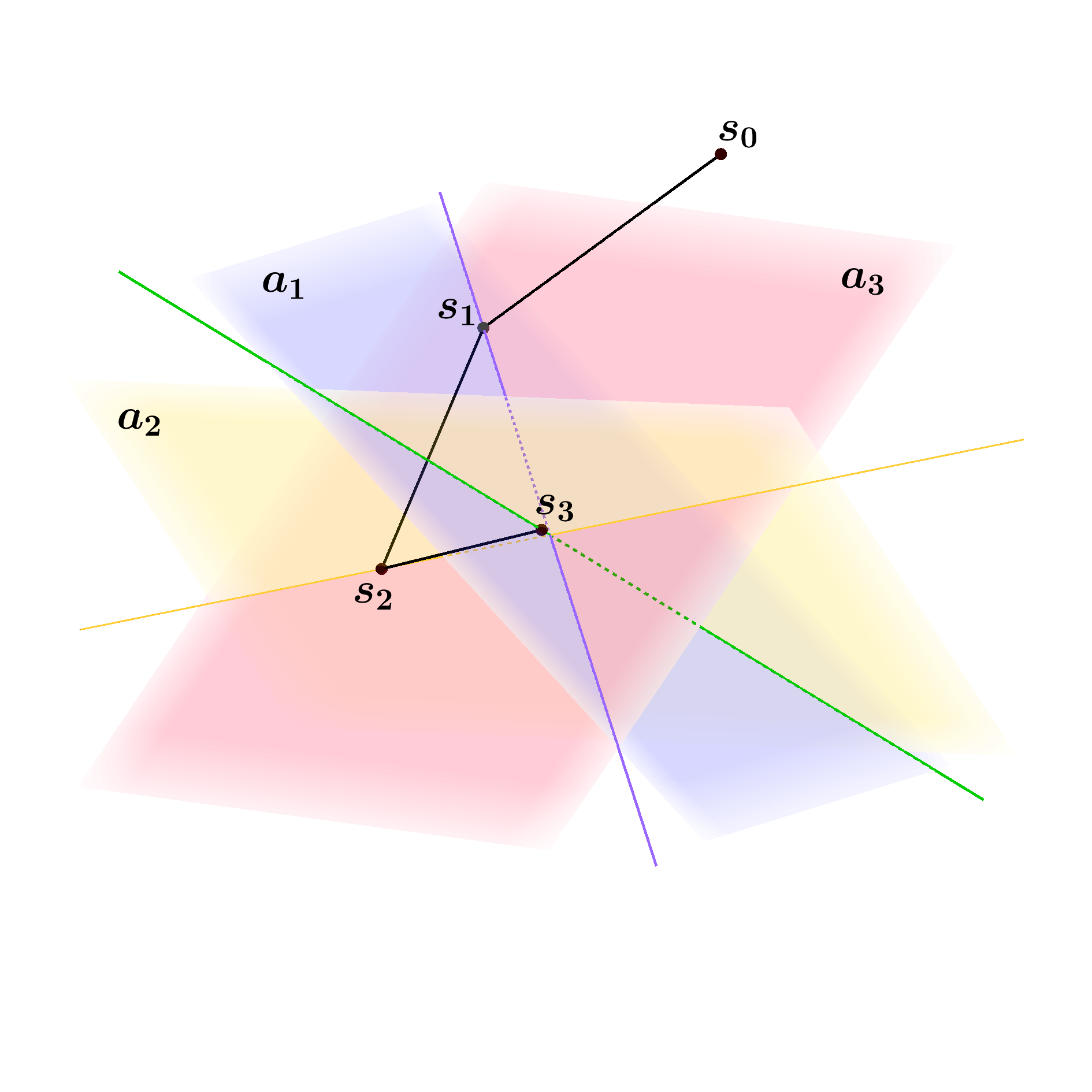}
      \caption{BRK update}
      \label{fig:updatebrk}
    \end{subfigure}
    \caption{Visualizations of updates. Linear equations are represented by hyperplanes and the solution space of two equations is represented by their intersection. \textbf{(Left)} An ALS update solves the linear system in one step. \textbf{(Right)} Three BRK updates defined by orthogonally projecting the previous iterate onto the solution space of the two sampled equations.}
    \label{fig:updates}
\end{figure}

\subsection{Block Randomized Kaczmarz} \label{subsec:brk}

Similar to ALS, the BRK algorithm first randomly initializes $A$ and $S$ and performs a matrix-vector reduction for a sampled column of $S$. The BRK algorithm then departs from ALS by performing a subsequent vector-sample reduction for $|\tau_1| \leq m$ rows of $A$ and $X$. 

It then performs a BRK update,
\begin{equation}
    S_{:,i}^{(j+1)} = S_{:,i}^{(j)} + (A_{\tau_1,:})^{\dagger}(X_{\tau_1,i} - A_{\tau_1,:}S_{:,i}^{(j)}).
\end{equation}

We can similarly update $A$ by selecting a subset $|\tau_2| \leq n$ of the columns of $S$,
\begin{equation}
    A_{i,:}^{(j+1)} = A_{i,:}^{(j)} + (X_{i,\tau_2} - A_{i,:}S_{:,\tau_2})(S_{:,\tau_2})^{\dagger}.
\end{equation}

The pseudocode for this algorithm is given in Method \ref{alg:brk}. Note that Method \ref{alg:als} details the alternating matrix factorization algorithm whereas Method \ref{alg:brk} details a BRK subroutine which is a specific option for the LS solver in Method \ref{alg:als}. Furthermore, the subsets $\tau_1, \tau_2$ do not need to be contiguous, and BRK is equivalent to a Randomized Kaczmarz (RK) update when $|\tau_1| = |\tau_2| = 1$. 

In this paper we will focus on a uniform sampling method, where the rows of $A$ are picked with probability $\frac{1}{m}$ and columns of $S$ are picked similarly. We will refer to this method as uniform block randomized Kaczmarz (UBRK). 

Others have previously found that a weighted sampling method offers the benefit of quicker convergence, however we did not find success with it in this application. When testing its performance against UBRK, calculations required for a weighted sample resulted in a significantly longer runtime for little relative error improvement.

\begin{algorithm}
	\caption{Alternating Scheme}\label{ALS}
	\label{alg:als}
	\begin{algorithmic}[1]
		\Procedure{AlteratingScheme}{$X \in \mathbb{R}^{m \times n}, k, M, L$} 
		\State{initialize $A \in \mathbb{R}^{m \times k}, S \in \mathbb{R}^{k \times n}$}
		\For{$l = 1, \ldots, L$}
		\State{sample column index $i \in [n]$}
		\State{use $L$ iterations of LS solver to approximately solve for $\ve{s} \in \mathbb{R}^k$ in $$X_{:,i} \approx A \ve{s}$$}
		\State{replace $i$th column of $S$ with $\ve{s}$, $$S_{:,i} = \ve{s}$$}
		\vspace{0.1cm}
		\State{sample row index $j \in [m]$}
		\State{use $L$ iterations of LS solver to approximately solve for $\ve{a} \in \mathbb{R}^k$ in $$X_{j,:}^\top \approx S^\top \ve{a}$$}
		\State{replace $j$th row of $A$ with $\ve{a}$, $$A_{j,:} = \ve{a}^\top$$}
		\EndFor{}
		\Return{$A, S$}
		\EndProcedure
	\end{algorithmic}
\end{algorithm}

\begin{algorithm}
	\caption{Block Randomized Kaczmarz (BRK)}\label{BRK}
	\label{alg:brk}
	\begin{algorithmic}[1]
		\Procedure{BRK}{$A \in \mathbb{R}^{m \times k}, \ve{b} \in \mathbb{R}^m, r, L$}
		\State{initialize $\ve{y}_0 \in \mathbb{R}^k$}
		\For{$l = 1, \ldots, L$}
		\State{sample $\tau \subset [n]$ such that $|\tau| = r$}
		\State{$\ve{y}_l = (I - A_{\tau,:}^\dagger A_{\tau,:})\ve{y}_{l-1} + A_{\tau,:}^\dagger \ve{b}_{\tau}$}
		\EndFor{}
		\Return{$\ve{y}_L$}
		\EndProcedure
	\end{algorithmic}
\end{algorithm}

\section{Experimental Results} \label{sec:experiments}
In this section, we apply the discussed algorithms to both synthetic and real data. We first compare their performances and results in detail on a small-scale synthetic data matrix before conducting a more significant comparison on synthetic and real large-scale data. Our implementations of these methods utilize Python's Numpy library~\cite{numpy}, and pseudo-inverse calculations were avoided using Numpy's lstsq() function. The test matrices in all tests were randomly generated using Numpy as well. 

We used an Intel(R) Core(TM) i7-4770 CPU(3.40 GHz) with 16.0 GB of RAM running on Windows OS for relative error computations. Runtime calculations were performed using an Intel(R) Core(TM) i3-3225 CPU(3.30 GHz) with 6.0 GB of RAM running on GNU/Linux. 

\subsection{Small-scale Synthetic Data} 

We applied our matrix factorization method on a random matrix $X \in \mathbb{R}^{1000 \times 1000}$. We generate $X$ as the product of factor matrices $A \in \mathbb{R}^{1000 \times 50}$ and $S \in \mathbb{R}^{50 \times 1000}$, which were generated by sampling each entry from $\{0, 1, 2, 3\}$ with probabilities $\{0.97, 0.1, 0.1, 0.1\}$ and $\{0, 1\}$ with probabilities $\{0.99, 0.01\}$, respectively. 
The resulting matrix $X$ had a sparsity of $98.58\%$ and true rank of $50$. 

Our proposed iterative methods have four main parameters: factor rank, number of alternating iterations, number of subiterations, and block size. Because the ideal factor rank is known to be $50$, it will be fixed at that value for all small-scale tests. Note that ALS will be used as a baseline to gauge performance and will be fixed to employ one subiteration and the maximum possible block size ($m$ rows of $A$ and $n$ columns of $S$) for all tests. Each data point in our plots is the average of 10 trials.

The first question we address is how much of the data matrix needs to be sampled at each iteration in order to converge to a solution. To answer this, we compared the UBRK method at block sizes of $20\%$ increments and ALS. The results are shown in Figure \ref{fig:smallsynth}. Although a wide range of block sizes can be used, we excluded block sizes that were not large enough to converge to a local solution.

We can see that block sizes of about $40\%$ and above are sufficient to converge to a factorization that captures most of the original matrix data (Figure~\ref{fig:smallsynth} left). Furthermore, as block size increases, the rate of convergence increases and final relative error decreases. However, while larger block sizes lead to faster convergence, they also scale faster in terms of runtime (Figure~\ref{fig:smallsynth} right). As a result, while a block size of 100\% for UBRK achieves the same relative error as ALS, it requires more time to run. However, ALS is optimal for small datasets due to their manageable size allowing a better tradeoff of runtime for accuracy, as well as lower memory limitations. We will see in the large-scale experiments that UBRK has a computational advantage over ALS in larger dimensions. 

\begin{figure}[ht]
    \makebox[\textwidth][c]{\includegraphics[width=\textwidth]{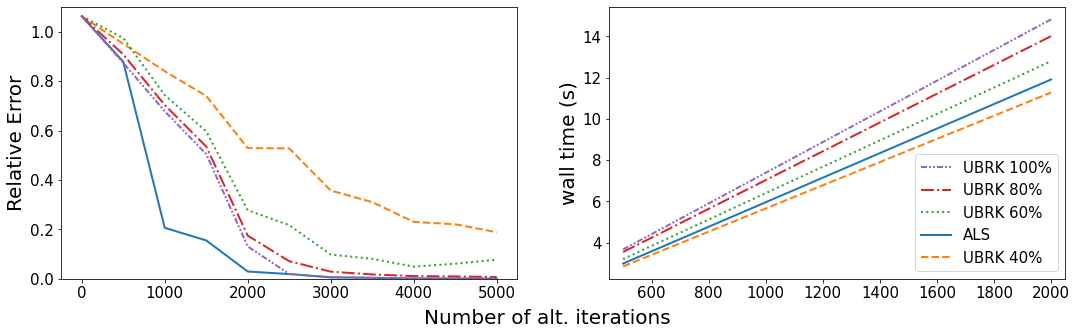}}%
    \caption{Results of small-scale synthetic data experiments. The data matrix $X \in \mathbb{R}^{1000 \times 1000}$ is generated by multiplying two synthetically generated factor matrices.
    The percentage following "UBRK" denotes the percentage of rows or columns sampled per iteration, i.e., UBRK 40\% means 40\% of the rows of A and 40\% of the columns of S were sampled by UBRK. 
    \textbf{(Left)} Relative error of methods with various numbers of alternating iterations. 
    \textbf{(Right)} Wall time of methods.}
    \label{fig:smallsynth}
\end{figure}

We also varied the number of subiterations amongst $1$, $10$, $20$, and $30$ while fixing the number of alternating iterations at $1000$. An exit condition was included when updating a single row or column; the method exited subiterations and continued to the next alternating iteration when $\norm{AS_{:, i} - X_{:, i}} < \epsilon$. We found that for block sizes greater than $1$, increasing the number of subiterations did not affect the relative error, meaning that the block Kaczmarz updates either converged after the first update or would have trouble converging at all.

\subsection{Large-scale Synthetic Data} We now turn to a large-scale data matrix more reflective of real world data, and increase the data matrix size to $X \in \mathbb{R}^{10^5 \times 1000}$. This matrix was generated similar to our previous experiment. The factor matrices $A \in \mathbb{R}^{10^5 \times 50}$ and $S \in \mathbb{R}^{50 \times 1000}$ were generated by sampling each entry from $\{0, 1, 2, 3\}$ with probabilities $\{0.97, 0.1, 0.1, 0.1\}$ and $\{0, 1\}$ with probabilities $\{0.999, 0.001\}$, respectively. This resulted in a slightly higher sparsity of $99.86\%$ and a true matrix rank of $29$, which did not significantly affect the relative error to which the algorithms converged. In light of the previous set of experiments, we will fix the number of subiterations at $1$ since increasing the subiteration count had no discernable effect on the resulting relative error. 

With a tall matrix, it may be useful to explore asymmetric block sizes - block sizes that are different for each factor matrix - to "balance" the optimality of both factor matrices.  The results of these experiments are depicted in Figure~\ref{fig:largesynth}. 

Note that ALS was terminated at $2000$ alternating iterations as its relative error effectively reached zero, and UBRK was terminated at $5000$ iterations for the same reason (Figure~\ref{fig:largesynth} left). While symmetric block sizes of $25\%$ and $50\%$ were slower to converge to a solution than ALS despite their faster runtimes, they can still be advantageous when facing a memory limitation, as they use $25\%$ and $50\%$ of the required memory for ALS, respectively. 

However, an interesting result is that an asymmetric block size of $1\%$ of rows of $A$ and the full $S$ matrix was comparable in performance to ALS (Figure~\ref{fig:largesynth} left).  This choice of block sizes not only cuts the runtime of the matrix factorization task in half, but also requires a bit less than $2\%$ of the working memory required by ALS.

It is also important, as evident in Figure~\ref{fig:largesynth} (left), to note that the runtime of block size pairs $25\% / 25\%$ (orange) and $1\% / 100\%$ (pink) are very similar despite the decrease from sampling $25\%$ to $1\%$ of the left factor matrix (or from $25,000$ rows to $1000$). This is due to the left factor requiring more updates than the right factor. Because our data matrix has $100$ times more rows than columns, the methods update the left factor matrix $100$ times for every $1$ update of the right factor. As a result the increase from sampling $25\%$ to $100\%$ of the right factor (or from $250$ columns to $1000$) overshadows the decrease in sampled rows. 

The algorithms have been able to achieve a low final relative error due to choosing a factor rank close to the true matrix rank of the synthetically generated data. When testing on real data, where the true rank is often unknown, our choice of factor rank may significantly affect the relative error of resulting approximations. This will be the case in the following experiment. 

\begin{figure}[ht]
    \makebox[\textwidth][c]{\includegraphics[width=\textwidth]{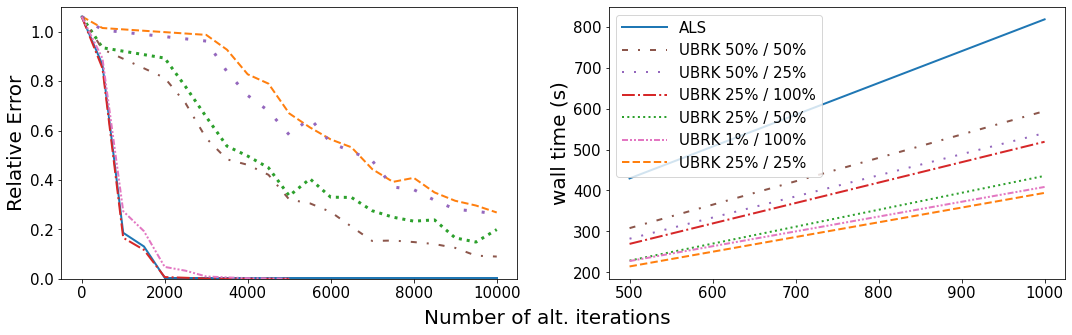}}%
    \caption{Results of large-scale synthetic data experiments. The data matrix $X \in \mathbb{R}^{1000 \times 1000}$ is generated by multiplying two synthetically generated factor matrices. 
    The percentages following "UBRK" denotes the percentage of rows/columns sampled per iteration, i.e. UBRK 25\%/50\% means 20\% of the rows of A and 50\% of the columns of S were sampled by UBRK. \textbf{(Left)} Relative error of methods with various numbers of alternating iterations. 
    \textbf{(Right)} Wall time of methods.}
    \label{fig:largesynth}
\end{figure}

\subsection{Real World Data} Finally, we compare algorithm performance on a real-world dataset from Amazon, which is publicly available as part of Dr. J. McAuley's Recommender Systems Datasets~\cite{he, ni}. The data was originally a dataframe of reviewer IDs, product IDs, time and ratings, which we reformatted by generating a data matrix with the reviewer IDs as the row indices, product IDs as the column indices, and ratings as the matrix entries. The resulting matrix was $128877 \times  1548$ and had a sparsity of $0.99926$. Zero entries represented the absence of a rating and nonzero integer entries ranged from $1$ to $5$. 

We once again chose a factor rank of $50$. However, the true matrix rank was $1548$, and this difference will impact the quality of resulting approximations, as shown in Figure \ref{fig:app}. Both ALS and UBRK converged to the solution at a slower rate and converged to a higher relative error at about $0.7$. This means that the solution was only able to capture about $30\%$ of the information provided by the test matrix. This is due in part to the mismatch between the factor rank and the true data rank.

Similar to the large-scale synthetic data, sampling the entire $S$ matrix and a small percentage of $A$ enabled UBRK to reach a solution with final relative error similar to that of ALS, and to converge far quicker. 
In this scenario, UBRK with less than maximal $S$ block size was not able to converge to final relative error similar to that of ALS.

\begin{figure}[ht]
    \makebox[\textwidth][c]{\includegraphics[width=\textwidth]{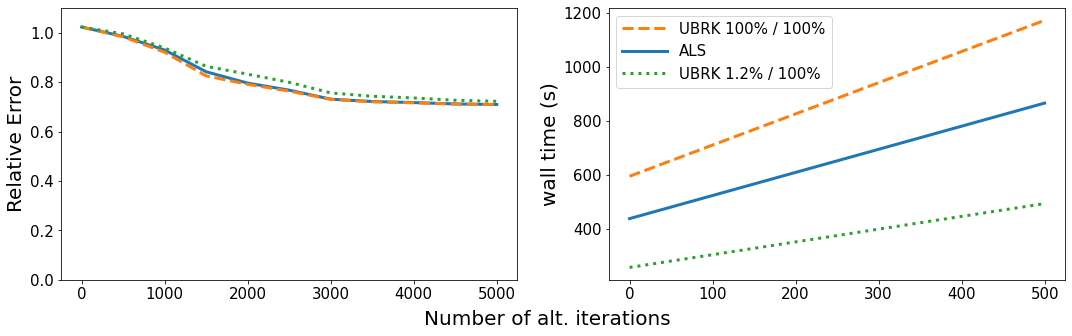}}%
    \caption{Results of Amazon dataset experiments. The data matrix $X \in \mathbb{R}^{128877 \times 1548}$ was full rank with a factor rank of $50$. \textbf{(Left)} Relative error of methods with various numbers of alternating iterations.
    \textbf{(Right)} Wall time of methods.}
    \label{fig:app}
\end{figure}

\section{Theoretical Results} \label{sec:theory}
One ideal property of iterative solvers for matrix factorization is that they do not update the factor matrices away from a (local or global) minimum of the objective function defining the problem formulation.  
Thus, because BRK is an iterative solver, it is of interest to bound the BRK iterative update at or near a stationary point to ensure it does not diverge too far from the local optimum. The following proposition provides an upper bound on the distance between the stationary point and a subsequent BRK iterative update.

\begin{proposition}
    Suppose $A^{(j)}, S^{(j)}$ locally minimize the objective function $F(A, S) = \frac{1}{2} \norm{X - AS}_F^2$ and $A^{(j+1)}$ has the $i^{th}$ row $A_{i,:}^{(j+1)}$, a BRK update of the $i^{th}$ row of $A^{(j)}$ using $S_{:, \tau}$. Then 
    $\norm{A^{(j+1)} - A^{(j)}}$ is bounded above by $\frac{\delta}{\sigma_{\text{min}}(SS^T)}$, where $\|\nabla F(A_{i, :}^{(j + 1)}, S)\| \leq \delta$.
\end{proposition} 

\begin{proof}
    Suppose $A^{(j)}$ and $S$ are minimizers. Then it is true that
    \begin{equation}
        \nabla_A F\left(A^{(j)}, S\right) = -S\left(X_{i, :}^T - S^T \left(A^{(j)}\right)^T\right) = 0.
    \end{equation}
    It immediately follows that 
    \begin{equation}
        \nabla_{A_{i, :}} F\left(A^{(j)}, S\right) = -S\left(X_{i, :}^T - S^T \left(A_{i, :}^{(j)}\right)^T\right) = 0. \label{eq:rowgrad}
    \end{equation}
        
    Now suppose $\norm{\nabla_{A_{i, :}} F\left(A^{(j+1)}, S\right)} \leq \delta$. Then
    \begin{equation}
        \norm{\nabla_{A_{i, :}} F\left(A^{(j)}, S\right) - \nabla_{A_{i, :}} F\left(A^{(j+1)}, S\right)} \leq \delta. \label{eq:graddist}
    \end{equation}
    
    Substituting in the gradient expression (\ref{eq:rowgrad}), we get
    \begin{equation}
        \norm{S\left(X_{i, :}^T - S^T \left(A_{i, :}^{(j)}\right)^T\right) - S\left(X_{i, :}^T - S^T \left(A_{i, :}^{(j+1)}\right)^T\right)} \leq \delta.
    \end{equation}
    
    Canceling terms, we're left with the simplified expression
    \begin{equation}
        \norm{SS^T\left(A_{i, :}^{(j+1)} - A_{i, :}^{(j)} \right)^T} \leq \delta.
    \end{equation}
    
    To recap, we are able to reduce the distance between the gradient of the objective function at $A_{i, :}^{(j)}$ and the updated $A_{i, :}^{(j+1)}$ denoted by (\ref{eq:graddist}) to an expression containing the difference between $A_{i, :}^{(j+1)}$ and $A_{i, :}^{(j)}$. This expression can be bounded by $\delta$ since the gradient at $A_{i, :}$ is assumed to be $0$. 
    
    Noting that $\norm{\left(A_{i, :}^{(j+1)} - A_{i, :}^{(j)}\right)^T} = \norm{I\left(A_{i, :}^{(j+1)} - A_{i, :}^{(j)}\right)^T}$ and applying the Cauchy-Schwarz inequality, we have
    \begin{equation}
        \norm{\left(A_{i, :}^{(j+1)} - A_{i, :}^{(j)}\right)^T} \leq \norm{\left(SS^T\right)^{\dagger}} \norm{SS^T \left(A_{i, :}^{(j+1)} - A_{i, :}^{(j)}\right)^T}. 
    \end{equation} 
    
    Assuming $SS^T \neq 0$, we can divide both sides by $\norm{\left(SS^T\right)^{\dagger}}$ to get
    
    \begin{equation}
        \frac{1}{\norm{(SS^T)^{\dagger}}} \norm{\left(A_{i, :}^{(j+1)} - A_{i, :}^{(j)}\right)^T} \leq \norm{SS^T \left(A_{i, :}^{(j+1)} - A_{i, :}^{(j)}\right)^T}.
    \end{equation}
    
    Because the spectral norm of a matrix is equal to its largest singular value $\sigma_{\text{max}}$ and a matrix's smallest singular value is its pseudoinverse's largest singular value, we reach the following result
    \begin{equation}
        \sigma_{\text{min}}\left(SS^T\right) \norm{\left(A_{i, :}^{(j+1)} - A_{i, :}^{(j)}\right)^T} \leq \delta. \label{eq:last}
    \end{equation}
\end{proof}

While $\delta$ may not always be sufficiently small, this will be the case if $A_{i, :}^{(j + 1)} S = X_{i, :}$, or if the approximation is at the global solution.

\section{Conclusions} \label{sec:conclusions}

This work rigorously compares the quality and speed of ALS and BRK subroutines in a common matrix factorization setting. We explore how different combinations of block sizes for BRK affect its solution quality and runtime when compared to an ALS benchmark. During this process we also observe patterns in particularly successful BRK block sizes; sampling most, or all, of the smaller factor matrix and only a small portion of the larger factor matrix is sufficient for converging to a solution of similar quality to ALS in less time. 

Directions for future work regarding Kaczmarz subroutines include finding optimal stepsizes and further theoretical proofs detailing relative error behavior as the solution approaches a local minimum. Future work on matrix factorization in general can also include finding an optimal factor matrix initialization. We also plan to generalize these techniques to tensor decompositions, where the size of the data can be extremely large and subsampling can offer drastic speedup.

\section*{Acknowledgements}
We would like to thank Jacob Moorman for his guidance and resources for preparing and publishing Python packages. Sponsor JH was partially supported by NSF DMS $\#2011140$ and NSF BIGDATA $\#1740325$.

\bibliographystyle{abbrv}
\bibliography{ref}

\end{document}